\documentclass[12pt]{article}
\usepackage{epsfig}
\usepackage{latexsym}
\usepackage{graphicx}
\usepackage{amsmath}
\usepackage{amsfonts}
\usepackage{amssymb}

\newtheorem{theorem}{Theorem}

\newtheorem{lemma}{Lemma}
\newtheorem{corollary}{Corollary}

\newtheorem{claim}{Claim}

\newcommand{\per}{{\rm per}}

\usepackage{color}

\newcommand{\qed}{\hfill\rule{0.5em}{0.809em}}

\def\bar{\overline}
\newenvironment{proof}{
\par
\noindent {\bf Proof.}\rm}{\mbox{}\hfill\rule{0.5em}{0.809em}\par}

\begin{document}
\title{Total weight choosability of $d$-degenerate graphs}
\author{
Tsai-Lien Wong \thanks{Department of Applied Mathematics, National
Sun Yat-sen University, Kaohsiung, Taiwan 80424. Grant numbers: MOST
104-2115-M-110 -001 -MY2. Email: tlwong@math.nsysu.edu.tw} \and
Xuding Zhu
\thanks{Department of Mathematics, Zhejiang Normal University,
China. Grant number:
  NSFC 11571319.  Email: xudingzhu@gmail.com. }
        \\[0.2cm]
       }

 %\date{2012.10.13}
% \date{2015.8.15}
%\date{2015.8.22}

\maketitle

\begin{abstract}

A graph $G$ is $(k,k')$-choosable if the following holds:
For any list assignment  $L$ which assigns to each vertex
$v$   a set $L(v)$ of $k$ real numbers, and assigns to each edge
$e$ a set $L(e)$ of $k'$ real numbers, there is a total weighting
 $\phi: V(G) \cup E(G) \to R$ such that $\phi(z) \in L(z)$ for $z \in V \cup E$,
and $\sum_{e \in E(u)}\phi(e)+\phi(u) \ne \sum_{e \in
E(v)}\phi(e)+\phi(v)$ for every edge $uv$. This paper proves the
following results: (1) If $G$ is a connected $d$-degenerate graph,
and $k>d$ is a prime number, and $G$ is either non-bipartite or has
two non-adjacent vertices $u,v$ with  $d(u)+d(v) < k$,   then $G$ is
$(1,k)$-choosable. As a consequence, every planar graph with no
isolated edges is $(1,7)$-choosable, and every connected
$2$-degenerate non-bipartite graph other than $K_2$ is
$(1,3)$-choosable. (2)   If $d+1$ is a prime number,   $v_1, v_2,
\ldots, v_n$ is an ordering of the vertices of $G$ such that each
vertex $v_i$ has back degree $d^-(v_i) \le d$, then there is a graph
$G'$ obtained from $G$ by adding at most $d-d^-(v_i)$ leaf
neighbours to $v_i$ (for each $i$) and $G'$ is $(1,2)$-choosable.
(3) If $G$ is $d$-degenerate and $d+1$ a prime, then $G$ is
$(d,2)$-choosable. In particular,  $2$-degenerate graphs are
$(2,2)$-choosable. (4) Every  graph is $(\lceil\frac{{\rm
mad}(G)}{2}\rceil+1, 2)$
% $(\lceil {\rm mad}(G)/2 \rceil+1, 2)$
-choosable. In
particular, planar graphs are $(4,2)$-choosable, planar bipartite
graphs are $(3,2)$-choosable.

\end{abstract}
{\small \noindent{{\bf Key words: }  Total weighting; $(k,k')$-choosable graphs;
permanent; $d$-degenerate graphs.}

\section{Introduction}

 A {\em total weighting} of a graph $G$ is a mapping
$\phi: V(G) \cup E(G) \to R$. A total weighting $\phi$ is {\em proper}
if for any edge $uv$ of $G$,
$$\sum_{e \in E(u)}\phi(e) + \phi(u) \ne \sum_{e\in E(v)}\phi(e) + \phi(v),$$
where $E(v)$ is the set of edges incident to $v$.
Total weighting of graphs has attracted considerable recent attention   \cite{KLT2004, Add2005,Add2007,WY2008,BGN09,KKP10,PW2010,PW2011,WZ11,ZW2012}.

The well-known 1-2-3 conjecture, proposed by Karo\'{n}ski, {\L}uczak
and  Thomason   \cite{KLT2004}, asserts that every graph with no
isolated edge has a proper total weighting $\phi$ with $\phi(v)=0$
for every vertex and $\phi(e) \in \{1,2,3\}$ for every edge $e$. The
conjecture has been studied by many authors
\cite{Add2005,Add2007,WY2008} and the current best result is that
the conjecture would be true if instead of   $\{1,2,3\}$, every edge
$e$ can have weight $\phi(e) \in \{1,2,3,4,5\}$ \cite{KKP10}. The
1-2 conjecture, proposed by    Przyby{\l}o and  Wo\'{z}niak in
\cite{PW2010}, asserts that every graph $G$ has a proper total
weighting $\phi$ with $\phi(z) \in\{1,2\}$ for all $z \in V(G)\cup
E(G)$. The best result on this conjecture
% was obtained by  Kalkowski in \cite{K2008}, where it was proved
is that every graph $G$ has a proper
total weighting $\phi$ with $\phi(v) \in \{1,2\}$ for $v \in V(G)$
and $\phi(e) \in \{1,2,3\}$ for $e \in E(G)$ \cite{K2008}.

Total weighting of graphs is naturally extended to the list version,
independently  by Przyby{\l}o and Wo\'{z}niak \cite{PW2011}  and by
Wong and Zhu \cite{WZ11}. Suppose $\psi: V(G) \cup E(G) \to
\{1,2,\ldots,\}$ is a mapping which assigns to each vertex and each
edge of $G$ a positive integer. A $\psi$-list assignment of $G$ is a
mapping $L$ which assigns to $z \in V(G) \cup E(G)$ a set $L(z)$ of
$\psi(z)$ real numbers. Given a total list assignment $L$, a proper
$L$-total weighting is a proper total weighting $\phi$ with $\phi(z)
\in L(z)$ for all $z \in V(G) \cup E(G)$. We say $G$ is {\em total
weight $\psi$-choosable} if for any $\psi$-list assignment $L$,
there is a proper $L$-total weighting of $G$. We say $G$ is
$(k,k')$-choosable if $G$ is $\psi$-total weight choosable, where
$\psi(v)=k$ for $v \in V(G)$ and $\psi(e) = k'$ for $e \in E(G)$.

As strengthenings of the 1-2-3 conjecture and the 1-2 conjecture, it
was conjectured in \cite{WZ11} that  every graph with no isolated
edges is $(1,3)$-choosable and every graph is $(2,2)$-choosable.
Some special graphs are shown to be $(1,3)$-choosable, such as
complete graphs, complete bipartite graphs, trees \cite{BGN09},
Cartesian product of an even number of even cycles, of  a path and
an even cycle,  of two paths \cite{WWZ2012}. Some special graphs are
shown to be $(2,2)$-choosable, such as complete graphs, generalized
theta graphs, trees \cite{WZ11}, subcubic graphs, Halin graphs
\cite{WZ2012}, complete bipartite graphs \cite{WYZ08}.

It was shown
in \cite{ZW2012} that every graph is $(2,3)$-choosable. However, it
is unknown whether there is a constant $k$ such that every graph with no
isolated edge is $(1,k)$-choosable, and whether there is a constant $k$ such that
every graph is $(k,2)$-choosable.

For graphs $G$ of maximum degree $k$ with no isolated edges,  it was proved by Seamone
\cite{Sea2014} that   $G$   is $(1, 2k+1)$-choosable, by Wang and Yan \cite{2014wy} that $G$ is
  $(1,\lceil\frac{4k+8}{3}\rceil$)-choosable, and recently,
 it is proved  in \cite{DDWWWYZ} that  $G$ is
$(1, k+1)$-choosable. In this paper, we first consider
connected $d$-degenerate graphs $G$. We prove that if $k > d \ge 2$ and either $G$ is non-bipartite
or $G$ is bipartite and there are two non-adjacent vertices
  $u,v$  with $d(u)+d(v) < k$,
then $G$ is $(1,k)$-choosable.
 As a consequence, every planar graph with no isolated edges is
$(1,7)$-choosable, and every connected $2$-degenerate non-bipartite
graph other than $K_2$ is $(1,3)$-choosable. Next we prove that if
$d+1$ is a prime number and $G$ is a $d$-degenerate graph, $v_1,
v_2, \ldots, v_n$ is an ordering of the vertices of $G$ such that
each vertex $v_i$ has back degree $d^-(v_i) \le d$, then there is a
graph $G'$ obtained from $G$ by adding at most $d-d^-(v_i)$ leaf
neighbours to $v_i$ (for each $i$) and $G'$ is $(1,2)$-choosable. In
particular, if $d+1$ is a prime number, and $G$ is a $d$-tree, then
for any $d$-clique $K$ of $G$, there is a graph $G'$ obtained from
$G$ by adding at most $j$   leaf neighbours to the $j$th vertex of
$K$ so that the resulting graph is $(1,2)$-choosable.

For $(k, 2)$-choosability,  we prove that if $G$ is
$d$-degenerate and $d+1$ a prime, then $G$ is $(d,2)$-choosable. In
particular, $2$-degenerate graphs are $(2,2)$-choosable.  In the
last section, we prove that
every  graph is $(
\lceil{mad(G)/2}\rceil +1, 2)$-choosable. In particular, planar
graphs are $(4,2)$-choosable, planar bipartite graphs are
$(3,2)$-choosable.

\section{$(1,k)$-choosability}

This section proves the  following result.
\begin{theorem}
\label{main}
Assume $G$ is a connected $d$-degenerate graph, $k > d \ge 2$ is a prime number and
 one of the following holds:
\begin{itemize}
\item $G$ is non-bipartite.
 \item $G$ is bipartite,and there are two non-adjacent vertices $u, v$ with $d(u)+d(v) <k$.
\end{itemize}
Then $G$ is $(1,k)$-choosable.
\end{theorem}

For each $z  \in V(G) \cup E(G)$, let $x_z$
be a variable associated to $z$. Fix an arbitrary orientation $D$ of $G$.
Consider the
polynomial
$$P_G(\{x_z: z \in V(G) \cup E(G)\})
= \prod_{uv \in E(D)}\left(  \left(\sum_{e \in E(u)} x_e+ x_u\right)
- \left(\sum_{e \in E(v)} x_e+ x_v\right)\right).$$
Assign a real number $\phi(z)$ to the variable $x_z$, and view
$\phi(z)$ as the weight of $z$.
Let $P_G( \phi  )$ be the evaluation of the
polynomial at $x_z = \phi(z)$. Then $\phi$ is a proper total
weighting of $G$ if and only if $P_G( \phi) \ne 0$.
The question  is under what condition one can find an assignment $\phi$
for which $P_G( \phi) \ne 0$.

An {\em index function} of $G$ is a mapping $\eta$ which
assigns to each vertex or edge $z$ of $G$ a non-negative integer $\eta(z)$.
An index function $\eta$ of $G$ is {\em valid} if $\sum_{z \in V \cup E}\eta(z) = |E|$.
Note that $|E|$ is the degree of the polynomial
$P_G(\{x_z: z \in V(G) \cup E(G)\})$.
For a valid index function $\eta$, let
$c_{\eta}$ be the coefficient of the monomial
 $\prod_{z \in V \cup E} x_{z}^{\eta(z)}$ in the expansion of $P_G$.
It follows from the Combinatorial Nullstellensatz
\cite{nullstellensatz,AlonTarsi} that if $c_{\eta} \ne 0$, and
$L$ is a list assignment which assigns to each $z \in V(G) \cup E(G)$ a
set $L(z)$ of  $ \eta(z)+1$ real numbers, then  there exists a
mapping $\phi$ with  $\phi(z) \in L(z)$   such that $$P_G(\phi) \ne
0.$$
An index function $\eta$ of $G$ is called {\em non-singular} if there is a valid index function
$\eta' \le \eta$ (i.e., $\eta'(z) \le \eta(z)$ for all $z \in V(G) \cup E(G)$) such that
$c_{\eta'} \ne 0$.

The main result of this section, Theorem \ref{main}, follows from  Theorem \ref{main0}.
\begin{theorem}
\label{main0}
Assume $G$ is a connected $d$-degenerate graph, $k >d \ge 2$ is a prime number and
 one of the following holds:
\begin{itemize}
\item $G$ is non-bipartite.
 \item $G$ is bipartite,  and there are two non-adjacent vertices $u, v$ with $d(u)+d(v) <k$.
\end{itemize}
Then $G$ has a non-singular index function $\eta$ with $\eta(v) = 0
$ for $v \in V(G)$ and $\eta(e) \le k-1$ for $e \in E(G)$.
\end{theorem}

We write the polynomial $P_G(\{x_z: z \in V(G) \cup E(G)\})$  as
$$P_G(\{x_z: z \in V(G) \cup E(G)\}) = \prod_{e \in E(D)}\sum_{z \in V(G) \cup E(G)}A_G[e,z]x_z.$$
It is straightforward to verify that for $e \in E(G)$ and $z \in V(G) \cup E(G)$,
if $e=(u,v)$ (oriented from $u$ to $v$), then
\begin{equation*}
A_G[e,z]=
\begin{cases} 1 & \text{if $z=v$, or $z \ne e$ is an edge incident to $v$,}
\\
-1 & \text{if $z=u$, or $z \ne e$ is an edge incident to $u$,}
\\
0 &\text{otherwise.}
\end{cases}
\end{equation*}
Now $A_G $ is a matrix, whose rows  are indexed by   edges of $G$
and the columns are indexed by edges and vertices of $G$. Given a
vertex or an edge $z$ of $G$, let $A_G(z)$ be the column of $A_G$
indexed by $z$.  As observed in \cite{WZ11}, for an edge $e=uv$   of $G$, we have
$$A_G(e)=A_G(u)+A_G(v). \eqno(1)$$

For an index function $\eta$ of $G$, let $A_G(\eta)$ be the
 matrix, each of
its column is a column of $A_G$, and each column $A_G(z)$ of $A_G$
occurs $\eta(z)$ times as  a column of $A_G(\eta)$.
For   $e \in E(G)$ and $z \in E(G) \cup V(G)$ with $\eta(z) \ge 1$,
$A_G[e,z]$ denote the entry of $A_G(\eta)$ at row $e$ and column $z$, and $A_G[\overline{e},\overline{z}]$
denotes the matrix obtained from $A_{G }(\eta)$ by deleting the row indexed by $e$ and  a column indexed by $z$.

It is known
\cite{alontarsi1989}
and easy to verify that for a valid
index function $\eta$ of $G$, $c_{\eta} \ne 0$ if and only if
$\per(A_G(\eta)) \ne 0$ (here $\per(A_G(\eta))$ denotes the permanent of $A_G(\eta)$).
Thus a valid index function $\eta$ of $G$ is
non-singular if and only if  $\per(A_G(\eta)) \ne 0$.

It is well-known (and follows easily from the definition) that
the permanent of a matrix is multi-linear on its column vectors (as well as its row vectors):
If a column   $C$ of  $A$ is
a linear combination of two columns vectors $C =\alpha C'+ \beta C''$,
 and $A'$ (respectively, $A''$) is obtained from $A$ by replacing the
column  $C$ with $C'$ (respectively,  with
$C''$), then
$${\rm per}(A) = \alpha {\rm per}(A') + \beta {\rm per}(A''). \eqno(2)$$

Assume $A$ is a square matrix whose columns are linear combinations
of columns of $A_G$. Define an index function $\eta_A: V(G) \cup
E(G) \to \{0,1, \ldots, \}$ as follows:

For $z \in V(G) \cup E(G)$,  $\eta_A(z)$ is  the number of columns
of $A$ in which $A_G(z)$
appears in the linear combinations with nonzero coefficient.

Note that the columns of $A_G$ are not linearly independent. There are different ways
of expressing the columns of a same matrix $A$
as linear combination of columns of $A_G$. So $\eta_A$ is not uniquely determined by the matrix
$A$ itself, instead it depends on
how its columns are expressed as linear combinations of columns of $A_G$.
For simplicity, we use the notation $\eta_A$, and each time the function $\eta_A$ is used,
it refers to an explicit expression of the columns of $A$ as linear combinations of columns of $A_G$.
In particular, for an index function $\eta$  of $G$, we may write a column of $A_G(\eta)$ as a linear combination
of other columns of $A_G$, and  $\eta_{A_G(\eta)} $ may become another index function of $G$.

To prove that a graph is $(1,k)$-choosable, it suffices to find a square matrix
$A$ with $\per(A) \ne 0$ whose columns are linear combinations of columns of $A_G$ such
that for each $v \in V(G)$, ${\eta}_A(v) = 0$, and for each edge
$e$ of $G$, ${\eta}_A(e) \le k-1$.

\begin{lemma}
\label{key}
Assume $G$ is a connected $d$-degenerate graph, $k >d \ge 2$ is a prime number and
 one of the following holds:
\begin{itemize}
\item $G$ is non-bipartite.
 \item $G$ is bipartite and there are two non-adjacent vertices $u, v$ with $d(u)+d(v) <k$.
\end{itemize}
Then there is a matrix $A$ whose columns are integral linear combinations (i.e., linear combination with integer coefficients)
of edge columns of $G$ such that
$\per(A) \ne 0 \pmod{k}$.
\end{lemma}

Before proving Lemma \ref{key}, we first show that Theorem
\ref{main0} follows from Lemma \ref{key}. Assume there is a matrix
$A$ whose columns are linear combinations of edge columns of $G$
such that $\per(A) \ne 0 \pmod{k}$. By repeatedly using (2), we know that there is a matrix $A'$
whose columns are edge columns of $G$ and $\per(A') \ne 0 \pmod{k}$.
If each edge column occurs at most $k-1$ times in $A'$, then we are
done. If there is an edge column which appears $k'$ times for some
$k' \ge k$, then $\per(A')$ is a multiple of $k'!$, and hence
$\per(A') = 0 \pmod{k}$, contrary to our choice of $A'$. This proves
that Theorem \ref{main0} follows from Lemma \ref{key}.

\noindent {\bf Proof of Lemma \ref{key}} First we consider the case
that $G$ is non-bipartite. Since $G$ is  a   $d$-degenerate graph,
there is an ordering  $v_1, v_2, \ldots, v_n$
 of the vertices such that for each $i$, vertex $v_i$ has $d^-(v_i) \le d$ neighbours
$v_j$ with $j < i$.  Let $A$ be the square matrix which consists of
$d^-(v_i)$ copies of $2A_G(v_i)$. It can be proved easily by induction on $n$ that $|\per(A)|=
2^m\prod_{i=1}^n d^-(v_i)!$. As $G$ is non-bipartite, we know that $d \ge
2$ and hence $k >2$. Also by our hypothesis, $d^-(v_i) \le d < k$ for
each $i$. Hence $\per(A) \ne 0 \pmod{k}$.

It suffices to show that  each column of $A$ is
an integral linear combination of edge columns of $G$.
In other words, for each vertex $v$ of $G$,
  $2 A_G(v)$ can be written as an integral linear combination of edge columns of $G$.

By assumption $G$ is connected and has an odd cycle $(u_0, e_0,u_1,
e_1, \ldots, u_{2q}, e_{2q}, u_0)$. If $v$ is on the cycle, say
$v=u_0$, then $2A_G(u_0)= A_G(e_0)-A_G(e_1)+A_G(e_2)-\ldots +
A_G(e_{2q})$. If $v$ is not on the odd cycle, then let $(w_0, e'_0,
w_1, e'_1, \ldots, e'_{t-1}, w_t)$ be a path connecting $v$ to
$u_0$, say $w_0=v$ and $w_t=u_0$. Then $2A_G(w_0) = 2A_G(e'_0) -
2A_G(e'_1) + 2A_G(e'_2) - \ldots \pm 2A_G(e'_{t-1}) \mp 2A_G(w_t)$,
and then write $2A_G(w_t)$ as an integral linear combination of edge
columns of $G$, we are done. This prove the non-bipartite case of
Lemma \ref{key}.

Assume $G$ is bipartite, and $u, v$ are the two specified vertices, and $d' =d(u)+d(v)$.
Similarly as above,   there is an ordering  $v_1, v_2, \ldots, v_{n-2}$
 of the vertices of $G-\{u, v\}$ such that for each $i$, vertex $v_i$ has $d^-(v_i) \le d$ neighbours
$v_j$ with $j < i$. Let $u=v_{n-1}, v=v_n$.  Let $A$ be the matrix
which consists of $d^-(v_i)$ copies of $A_G(v_i) \pm A_G(v) $ for
$i=1,2,\ldots, n-2$
 and $d'$ copies of $A_G(u) \pm A_G(v)$, where the $\pm$ is determined by the distance between the two involved vertices: if
 the distance is odd, then choose $+$, and otherwise choose $-$.
It is easy to verify that $|\per(A)|= (\prod_{i=1}^{n-2} d^-(v_i)!) d'!$.
Hence $\per(A) \ne 0 \pmod{k}$.

  It suffices to show that  each column   of $A$  can be written as an integral linear combination of edge columns of $G$.
  This is so, because if $x,y$ are two vertices connected by a path of odd
  length  $(u_0, e_0,u_1, e_1, \ldots, u_{2q}, e_{2q}, u_{2q+1})$, say $x=u_0, y=u_{2q+1}$,
  then $A_G(x)+A_G(y)=A_G(e_0)-A_G(e_1)+\ldots +A_G(e_{2q})$. If $x,y$ are two vertices connected by a path of even
  length  $(u_0, e_0,u_1, e_1, \ldots, u_{2q-1}, e_{2q-1}, u_{2q})$, say $x=u_0, y=u_{2q}$,
  then $A_G(x)-A_G(y)=A_G(e_0)-A_G(e_1)+\ldots -A_G(e_{2q-1})$. This completes the proof of Lemma \ref{key}.
  \qed

\begin{corollary}\label{cor1}
If $G$ is  $d$-degenerate, non-bipartite graph, then  $G$ is $(1,
2d-3)$-choosable.
\end{corollary}
\begin{proof}
Using the Bertrand Theorem that for $d > 3$, there is a prime $p$
such that $d < p < 2d-2$.
\end{proof}

\begin{corollary}
If $G \ne K_2$ is a tree, or  a $2$-tree, then $G$ is
$(1,3)$-choosable. If $G$ is a $3$-tree, then $G$ is
$(1,5)$-choosable. If $d \geq 4$ and $G$ is a $d$-tree, then $G$ is
$(1, 2d-3)$-choosable.
\end{corollary}
\begin{proof}
All these   follow easily from Theorem \ref{main} and Corollary
\ref{cor1}
\end{proof}

The result that trees are
$(1,3)$-choosable was proved in \cite{BGN09}, however, the proof is
different from the one presented here.

\begin{corollary}
\label{main2} Every planar graph with no isolated edges is $(1,7)$-choosable.
\end{corollary}
\begin{proof}
We may assume $G$ is connected, for otherwise, we consider
components of $G$ separately. It is well-known that every planar
graph is $5$-degenerate. If $G$ is non-bipartite, then we are done
by Theorem \ref{main}. If $G$ is bipartite, then $G$ is triangle
free. By Euler formula $G$ has minimum degree $\delta(G) \le 3$.
If $\delta(G)=3$, then it follows from Euler formula that $G$ has at least $8$ vertices of
degree $3$, and hence there are non-adjacent vertices $u$  and $v$
with $d(u)+d(v) <7$.  In case $\delta(G)=1$ or $2$, it is also easy to see that there are two non-adjacent
vertices $u,v$ with $d(u)+d(v) <7$. So the conclusion again follows
from Theorem \ref{main}.
\end{proof}

%\begin{corollary}
%\label{maxidegree}
%If $G$ is a   graph with no isolated edges, and for each edge $e=uv$, $d(u)+d(v) \le k$, then
%$G$ is $(1, k-1)$-choosable.
%\end{corollary}
%\begin{proof}
%It follows from Lemma \ref{key} that $G$ has a non-singular index function $\eta$ for which $\eta(v)=0$ for every vertex $v$.
%For an edge $e=uv$, the edge column $A_G(e)$ has $d(u)+d(v)-2$ nonzero entries. So if $\eta$ is an index function for which
%$A_G(\eta)$ is a square matrix and $\eta(e) > d(u)+d(v)-2$, then $\per(A_G(\eta)) = 0$. Therefore $G$ has a non-singular index function
%$\eta$ for which $\eta(v)=0$ for each vertex $v$, and $\eta(e) \le d(u)+d(v)-2$ for each edge $e=uv$.
%\end{proof}

%Note that in Corollary \ref{maxidegree}, the integer $k$ is not necessarily a prime.

\section{Almost $(1,2)$-choosability}

In this section, we prove the following result.

\begin{theorem}
\label{12choosable}
Assume   $d+1$ is a prime number and $G$ is a $d$-degenerate graph.
Let $v_1, v_2, \ldots, v_n$ be an ordering of the vertices of $G$ such that
each vertex $v_i$ has $d^-(v_i) \le d$ backward neighbours. Then there is a $(1,2)$-choosable graph $G'$ obtained from $G$ by adding
at most $d-d^-(v_i)$ leaf neighbours to $v_i$ (i.e., neighbours of degree $1$).
\end{theorem}

Prior to this paper, all the known $(1,2)$-choosable graphs are bipartite graphs.
As a consequence of this lemma,  every graph $G$
is a subgraph of a $(1,2)$-choosable graph $G'$.

Before proving Theorem \ref{12choosable}, we shall first prove that
if $G$ is $d$-degenerate and each vertex of $G$ has backdegree
``almost" $d$, then $G$ is ``almost" $(1,2)$-choosable.

\begin{lemma}
\label{subgraph} Assume $G$ is a graph and $\eta$ is a non-singular
index function of $G$, and $E'$ is a subset of edges of $G$.
If $\eta(e)=0$ for every $e \in E'$, then $\eta$ is a non-singular index function of $G-E'$.
\end{lemma}
\begin{proof}
Let $G'=G-E'$. As $\eta(e)=0$ for every $e \in E'$, $A_{G'}(\eta)$  is the matrix obtained from  $A_G(\eta)$ by deleting the rows indexed
by edges $e \in E'$.  Since $\per(A_G(\eta)) \ne 0$, one can delete some columns from $A_{G'}(\eta)$ to obtain
a square matrix with nonzero permanent. I.e.,
there is a valid index function $\eta'$ of $G'$  such that   $\eta' \le \eta$,
and $\per(A_{G'}(\eta')) \ne 0$. Thus   $\eta$ is a non-singular index function of $G'$.
\end{proof}

\begin{theorem}
\label{mixed}
Assume $d+1$ is a prime number, $G$ is a $d$-degenerate graph, and  $v_1, v_2, \ldots, v_n$ is an ordering
 of the vertices such that for each $i$, vertex $v_i$ has $d^-(v_i) \le d$ backward neighbours.
 Let $G'$ be obtained from $G$ by adding $d-d^-(v_i)$ leaf neighbours to $v_i$ for $i=1,2,\ldots, n$.
 Let $\eta$ be the index function of $G'$ defined as $\eta(v_i)=d-d^-(v_i)$ and $\eta(e)=1$ for each edge $e$ of $G$,
 and $\eta(z)=0$ for each added vertex and edge.
Then $\per(A_{G'}(\eta)) \ne 0 \pmod{d+1}$.
\end{theorem}
\begin{proof}
Let $M_0=A_{G'}(\eta)$.
For $i=1,2,\ldots, n$, let $M_i$ be obtained from $M_{i-1}$ as
follows: For each edge $e = v_iv_j \in E(G)$ with $j <i$, replace
the edge column $A_{G'}(e)$ with $A_{G'}(v_i)$.

\begin{claim}\label{cl1}
For any $j \le i$, $M_i$ contains exactly $d$ copies of the column
$A_{G'}(v_j)$ and $\per(M_i) = \per(M_{i-1}) \pmod{d+1}$.
\end{claim}

First we prove that $M_i$ contains exactly $d$ copies of the column
$A_{G'}(v_j)$ for any $j \le i$. This is certainly true for $i=1$, because $M_1=M_0$ and $\eta(v_1)=d$. Assume this is true for $M_{i-1}$.
By the rule above, $d^-(v_i)$ copies of $A_{G'}(v_i)$ are used to replace $d^-(v_i)$ edge columns. Since $M_{i-1}$
has $\eta(v_i)=d-d^-(v_i)$ copies of $A_{G'}(v_i)$,  we conclude that $M_i$ contains exactly $d$ copies of
$A_{G'}(v_i)$ as its column vectors.

Now we   prove   $\per(M_i) = \per(M_{i-1}) \pmod{d+1}$. For each
edge $e=v_iv_j$   with $j <i$, we write the column   $A_{G'}(e)$ in
$M_{i-1}$ as $A_{G'}(v_i)+A_{G'}(v_j)$. Apply (2) to expand
$\per(M_{i-1})$ as the   sum of a family of  permanents. Then
$\per(M_i)$ is one of the permanents. For each of the other
permanents $M'$, there is an index $j < i$ such that $M'$ contains
at least one more column of $A_{G'}(v_j)$ than $M_{i-1}$, and hence
contains at least $d+1$ copies of the column $A_{G'}(v_j)$.
Therefore $\per(M') = 0 \pmod{d+1}$. Therefore $\per(M_i) =
\per(M_{i-1}) \pmod{d+1}$. This completes the proof of the claim.

Let $\eta'$ be the index function defined as $\eta'(v)=d$ for   $v \in V(G)$, $\eta'(v)=0$ for   $v \in V(G') \setminus V(G)$
 and $\eta'(e)=0$
for each edge $e$ of $G'$. By Claim \ref{cl1},   $M_n=A_{G'}(\eta')$. As each vertex $v \in V(G)$ has back degree exactly $d$,
we conclude that   $|\per(M_n)| = (d!)^n \ne 0 \pmod{d+1}$. Therefore $\per(A_{G'}(\eta)) \ne 0 \pmod{d+1}$. So $\eta$ is a non-singular index function
of $G'$. By Lemma \ref{subgraph}, $\eta$ is a non-singular index function of $G$.
\end{proof}

If $d+1$ is prime, $G$ is $d$-degenerate and almost every   vertex
has back degree exactly $d$, then $G$ is ``almost" $(1,
2)$-choosable. For example, we have the following corollary.

\begin{corollary}
\label{ktree}
If $d+1$ is a prime number and $G$ is a $d$-tree, then $G$ is almost $(1,2)$-choosable, except that the first $d$ vertices require lists of sizes
$d+1, d, \ldots, 2$, respectively. In particular, if $G$ is a tree, and $v$ is an arbitrary vertex of $G$, then $G$ is   $(1,2)$-choosable, except that $v$ needs a list of size $2$.
If $G$ is $2$-degenerate, and every vertex except the first $2$
vertices have back-degree exactly $2$, then $G$ is almost $(1,2)$-choosable,
except that for the first two vertices $v_1, v_2$ need a list of size $3$.
\end{corollary}

Now we are ready to prove Theorem \ref{12choosable}.
For a graph $G$, let $B_G =A_G(\eta)$, where $\eta(e)=1$ for each edge $e$, and $\eta(v)=0$ for each vertex $v$.

\begin{lemma}
 \label{subsub}
 Assume $\eta$ is an index function of a graph $G$ and $X$ is a set of leaves of $G$ for which the following hold:
 \begin{enumerate}
 \item For each edge $e$, $\eta(e) =0$ if $e$ is incident to a vertex in $X$ and  $\eta(e)=1$ otherwise.
 \item For each vertex $v$, $\eta(v) = |N_G(v) \cap X|$.
 \end{enumerate}
 If $\per(A_G(\eta)) \ne 0$, then there is a subset $Y$ of $X$ such that $G-Y$ is $(1,2)$-choosable.
 \end{lemma}
\begin{proof}
Assume the lemma is not true and $G$ is a minimum counterexample.

For each vertex $v$ of $G$, let   $N_G(v) \cap X = \{v'_j: 1 \le j \le \eta(v)\}$   and let $e_{v,j} = vv'_j$.
Take the matrix $B_G$, and for each edge $e_{v,j}$,   write $A_G(e_{v,j})$ as the sum $A_{G}(v)+A_G(v'_j)$.
By repeatedly using (2), $\per(B_G)$ can be written as the summation of the permanents of many matrices.
To be precise,
$\per(B_G) = \sum_{\eta' \in \Gamma} A_G(\eta')$, where $ \Gamma$  consists of all the index functions $\eta'$ such that
\begin{enumerate}
 \item $\eta'(e) = \eta(e)$ for  each edge $e$.
 \item For $v'_j \in X$, $\eta'(v'_j) = 0$ or $1$.
 \item For each vertex $v$,
 $\eta'(v) = \eta(v) - | \{v'_j: \eta'(v'_j) = 1\}|$.
 \end{enumerate}

Observe that $\eta \in \Gamma$.

\begin{claim}
\label{clmn}
If $\eta' \in \Gamma$ and $\eta' \ne \eta$, then $\per(A_G(\eta')) = 0$.
\end{claim}
\begin{proof}
Assume to the contrary that there exists
 $\eta' \in \Gamma$,  $\eta' \ne \eta$, $\per(A_{G}(\eta')) \ne 0$.

Let $Z =\{v'_j \in X: \eta'(v'_j) =1\}$. As $\eta' \ne \eta$, $Z \ne
\emptyset$. The column $A_G(v'_j)$ has only one entry equals $1$,
namely the entry at the row indexed by $e'_{v,j}$, and all the other
entries are $0$. Therefore, $\per(A_{G-Z}(\eta'))=\per(A_G(\eta'))$,
where in $\per(A_{G-Z}(\eta'))$,   $\eta'$ denotes its restriction
to $G-Z$. As $\per(A_{G-Z}(\eta')) \ne 0$,
  $G'=G-Z$ together with $\eta'$ and $X'=X-Z$ satisfy the condition of Lemma \ref{subsub}.
By the minimality of $G$, there is a subset $Y'$ of $X'$, such that $G'-Y'$ is $(1,2)$-choosable.
Let $Y=Y' \cup Z$, we have $G-Y$ is $(1,2)$-choosable, a contradiction.
This completes the proof of Claim \ref{clmn}.
\end{proof}

Now Claim \ref{clmn} implies that $\per(B_G)=\per(A_G(\eta)) \ne 0$, and hence $G$ itself is $(1,2)$-choosable,
a contradiction.
 \end{proof}

Theorem \ref{12choosable} follows    from Theorem \ref{mixed} and Lemma \ref{subsub}.

\begin{corollary}
\label{ktree}
If $d+1$ is a prime number, $G$ is $d$-tree, and $K$ is a $d$-clique in $G$, then there is a $(1,2)$-choosable graph which is
obtained from $G$ by adding $k_1,k_2,\ldots, k_d$ leaf neighbours to the $d$ vertices of $K$ respectively, for some $k_j \le j$.
\end{corollary}
\begin{proof}
The vertices of $G$ can be ordered as $v_1, v_2, \ldots, v_n$ so that $K=\{v_1,v_2,\ldots, v_d\}$ and
$v_j$ has $j-1$ backward neighbours for $j \le d$, and each other vertex has $d$ backward neighbours. The conclusion then follows from
Theorem \ref{12choosable}.
\end{proof}

 The $d=1$ case of Corollary \ref{ktree} was proved in \cite{cwz}, where it is shown
 that trees with an even number of edges are
$(1,2)$-choosable.

\section {$(k, 2)$-choosability }

By applying  Theorem \ref{mixed}, we prove in this section that when $d+1$ is a prime, then
$d$-degenerate graphs are $(d,2)$-choosable. In particular, $2$-degenerate graphs are $(2,2)$-choosable.

\begin{theorem}
\label{ddeg}  Assume $d+1$ is a prime number, $G$ is a
$d$-degenerate graph. Then $G$ is $(d,2)$-choosable.
\end{theorem}
\begin{proof}
Assume Theorem \ref{ddeg} is not true, and $G$ is a connected
$d$-degenerate graph which is not $(d,2)$-choosable.  Let
  $v_1, v_2, \ldots, v_n$ be an ordering of the vertices of $G$
  such that  $1 \le d^-(v_i) \le d$ for $2 \le i \le n$ (note that $d^-(v_1)=0$).
  %Here $d^-(v_i)$ is the number of {\em backward neighbours} of $v_i$, i.e., neighbours  $v_j$ of $v_i$ with $j < i$.
Let $G'$ by obtained from $G$ by adding $d-d^-(v_i)$ leaf neighbours
to $v_i$ for $i=1,2,\ldots, n$.

Let $\eta$ be the index function of $G'$ such that $\eta(e)=1$ for
every edge $e$ of $G$, $\eta(v_i)=d-d^-(v_i)$ for every vertex of
$G$, and $\eta(z)=0$ for all the added vertices and edges $z$. Since
$1 \le d^-(v_i) \le d$ for $i>1$, hence $\eta(v_i)\leq d-1$ for
every vertex of $G$ except that $\eta(v_1)=d$.

Since $|E(G')|=|E(G)|+\sum_{i=1}^n (d-d^-(v_i))$,  $A_{G'}(\eta)$ is
a
square matrix. By %the proof of
Theorem \ref{mixed}, we have $\per(A_{G'}(\eta)) \ne 0 \pmod{d+1}$.

It follows from Lemma \ref{subgraph} that there is a non-singular
index function $\eta'$ of $G$ with $\eta'(z) \le \eta(z)$ for $z \in
V(G) \cup E(G)$. In the following, we shall further prove that there
is such an index function $\eta'$ for which $\eta'(v_1)$ is strictly
less than $\eta(v_1)$. Hence   $\eta'(z) \le d-1$ for all $z \in
V(G)$ and $\eta'(z) \le 1$ for all $z \in E(G)$ and hence $G$ is
$(d,2)$-choosable, which is in contrary to our assumption.

We define a {\em comb-plus subgraph} of $G'$ as a subgraph indicated
in Figure \ref{figcomb}, where   $(w_1, w_2, \ldots, w_p)$ is a path
in $G$,  $w_p$ adjacent to $w_s$ for some $1 \le s \le p-2$, and
$e'_j=w_ju_j \in E(G')-E(G)$ for $j=1,2,\ldots, p$.

\begin{figure}[ht]
\begin{center}
\scalebox{0.36}{\includegraphics{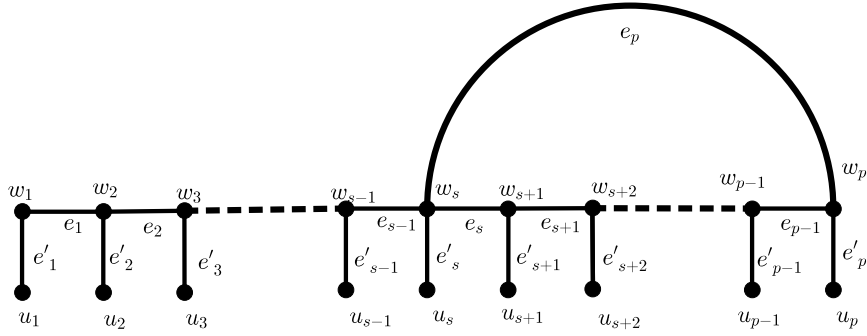}}
\caption{The comb-plus
subgraph}
\label{figcomb}
\end{center}
\end{figure}

 \begin{claim}
 \label{clm}
There is a  comb-plus subgraph of $G'$ as in Figure \ref{figcomb}
for which the following hold:
  \begin{itemize}
 \item $\eta(w_1)=d$ and  $\eta(w_j)=d-1$ for $2 \le j \le p$ and $\eta(e_j)=1$ for $1 \le j \le p $, where $e_j=w_jw_{j+1}$ for $1 \le j \le p-1 $
 and $e_p=w_pw_s$.
 \item For $0 \le i \le p$, $\per(A_{H_i}(\eta_i)) \ne 0 \pmod{d+1}$, where   $H_i=G'-\{e'_1,e'_2,\ldots, e'_i\}$, and $\eta_i = \eta$, except that $\eta_i(e_j)=0$ for $1 \le j \le i$.
 \end{itemize}
 \end{claim}

\begin{proof}
We choose the vertices $w_1, w_2, \ldots, w_p$, and hence the edges
$e'_1, e_1, e'_2, e_2,\ldots, e'_p, e_p$, recursively. Initially let
$w_1=v_1$. Let $e'_1=w_1u_1$ be an added edge incident to $w_1$
(recall that $v_1$ is incident to $d$ added  edges). Note that
$H_0=G'$ and $\eta_0=\eta$.

Calculating $\per(A_{H_0}(\eta_0))$ by expanding along the row
indexed by $e'_1$, we conclude that there is a column of
$A_{H_0}(\eta_0)$ indexed by $z \in V(G) \cup E(G)$  such that
$$A_{H_0}(\eta_0)[e'_1, z] \ne 0 \pmod{d+1} \ \ \mbox{\rm and} \ \  \per(A_{H_0}(\eta_0)[\overline{e'_1}, \overline{z}]) \ne 0 \pmod{d+1}.$$
As $A_{H_0}(\eta_0)[e'_1, z] \ne 0$, we know that either $z=w_1$ or
$z$ is an edge of $G$ incident to $w_1$.

%%%%%%%
Note that $H_1=H_0-e'_1$, hence
 $$A_{H_0}(\eta_0)[\overline{e'_1}, \overline{z}] =A_{H_1}(\eta_1)$$ where $\eta_1$   agrees with
 $\eta_0$, except that $\eta_1(z) = \eta_0(z)-1$.

If $z=w_1$,  then  $\eta_1(w_1) = d-1$. It follows from Lemma
\ref{subgraph} that there is a non-singular index function $\eta'$
for which $\eta'(z ) \le d-1$ for all $z \in V(G) $ and $\eta'(z )
\le 1$ for all $z \in  E(G)$, and hence $G$ is $(d,2)$-choosable,
contrary to our assumption.

Assume $z$ is an edge of $G$ incident to $w_1$. Let $w_2$ be the
other end vertex of $z$, and let $e_1= z=w_1w_2$. If $\eta(w_2)\leq
d-2$, then  write the column $A_{H_1}(w_1)$ of $A_{H_1}(\eta_1)$ as
$A_{H_1}(e_1)-A_{H_1}(w_2)$.
 %=====
   By this
expression of the matrix $A_{H_1}(\eta_1)$, we have
$\eta_{A_{H_1}(\eta_1)}(z) \le d-1$ for all $z \in V(G)$ and
$\eta_{A_{H_1}(\eta_1)}(z) \le 1$ for all $z \in E(G)$ and
$\eta_{A_{H_1}(\eta_1)}(z) = 0$ for all $z \notin V(G) \cup E(G)$.
As $\per(A_{H_1}(\eta_1)) \ne 0$,  by Lemma \ref{subgraph}, $G$ is
$(d,2)$-choosable, contrary to our assumption.

Thus we may assume that $\eta(w_2)=d-1$.

Assume $i \geq 1$, and we have chosen distinct vertices $w_1, w_2,
\ldots, w_i$, edges $e'_1, e'_2, \ldots, e'_i$ and $e_1=w_1w_2,
e_2=w_2w_3, \ldots, e_i=w_iw_{i+1}$, for which the following hold:
\begin{itemize}
\item   $\eta(w_1)=d$ and
$\eta(w_j) = d-1$ for $2 \le j \le i+1$ and $\eta(e_j)=1$ for $1 \le
j \le i$.
\item  $\eta_j=\eta_{j-1}$ except that $\eta_j(e_j)=\eta_{j-1}(e_j)-1=0$ for $1 \le j\le i$.
\item  For $0 \le j \le i$, $\per(A_{H_j}(\eta_j)) \ne 0 \pmod{d+1}$.
\end{itemize}
If $w_{i+1}=w_s$ for some $1 \le s \le i-2$, then let  $p=i$, and
the claim is proved. Assume $w_{i+1} \ne w_j$ for any $1 \le j \le
i-2$. Since $\eta(w_{i+1})=d-1$, $d-d^-(w_{i+1})=d-1$ and there is
an edge $e'_{i+1}=w_{i+1}u_{i+1} \in E(G')-E(G)$. As $w_{i+1} \ne
w_j$ for any $1 \le j \le i-2$, we have $e'_{i+1} \in E(H_i)$.

Calculating $\per(A_{H_i}(\eta_i))$ by expanding along the row
indexed by $e'_{i+1}$, we conclude that there is a column of
$A_{H_i}(\eta_i)$ indexed by $z \in V(G) \cup E(G)$  such that
$$A_{H_i}(\eta_i)[e'_{i+1}, z] \ne 0 \pmod{d+1}, \     \per(A_{H_i}(\eta_i)[\overline{e'_{i+1}}, \overline{z}]) \ne 0 \pmod{d+1}.$$
Similarly, $A_{H_i}(\eta_i)[e'_{i+1}, z] \ne 0$ implies that either
$z=w_{i+1}$ or $z$ is an edge of $G$ incident to $w_{i+1}$.

 As $H_{i+1}=H_i-e'_{i+1}$, we have
 $$A_{H_i}(\eta_i)[\overline{e'_{i+1}}, \overline{z}] =A_{H_{i+1}}(\eta_{i+1})$$ where $\eta_{i+1}$ is an index function which agrees with
 $\eta_i$, except that $\eta_{i+1}(z) = \eta_i(z)-1$.

If $z=w_{i+1}$,   then  $\eta_{i+1}(w_{i+1}) = d-2$. In
$A_{H_{i+1}}$,
$$A_{H_{i+1}}(w_1)=A_{H_{i+1}}(e_1)-A_{H_{i+1}}(e_2)+A_{H_{i+1}}(e_3)-\ldots +(-1)^{i-1}A_{H_{i+1}}(e_{i})+(-1)^{i}A_{H_{i+1}}(w_{i+1}).$$
%the column indexed by $w_1=v_1$ is equal to
%$e_1-e_2+e_3- \ldots +(-1)^{i-1}e_i+(-1)^iw_{i+1}$.
By this expression of the columns of $A_{H_{i+1}}(\eta_{i+1})$, the
column $A_{H_{i+1}}(z)$ occurs at most $d-1$ times for each $z \in
V(G)$ and the column $A_{H_{i+1}}(z)$ occurs at most once for each
$z \in E(G)$. For each $z \notin V(G) \cup E(G)$,  the column
$A_{H_{i+1}}(z)$ does not occur. By Lemma \ref{subgraph}, $G$ is
$(d,2)$-choosable, contrary to our assumption.

Assume $z$ is an edge of $G$ incident to $w_{i+1}$. Let $w_{i+2}$ be
the other end vertex of $z$ and let $e_{i+1}=z=w_{i+1}w_{i+2}$.
 If $\eta_{i+1}(w_{i+2})\leq d-2$, then in $A_{H_{i+1}}(\eta_{i+1})$,
 %the column indexed by $w_1=v_1$ is equal to
$$A_{H_{i+1}}(w_1)=A_{H_{i+1}}(e_1)-A_{H_{i+1}}(e_2)+A_{H_{i+1}}(e_3)-\ldots +(-1)^iA_{H_{i+1}}(e_{i+1})+(-1)^{i+1}A_{H_{i+1}}(w_{i+2}),$$
%$e_1-e_2+e_3- \ldots +(-1)^{i}e_{i+1}+(-1)^{i+1}w_{i+2}$,
which again leads to a contradiction. Thus  $\eta_{i+1}(w_{i+2})
=d-1$.

This process of finding new vertices $w_j$ will eventually stop (as
$G$ is finite), and at the end we obtain the required comb-plus
subgraph. This completes the proof of Claim \ref{clm}.
\end{proof}

Assume first that $s=1$, and hence  $C= ( w_1 , w_2,   \cdots, w_p)$
is a cycle. By definition, $\eta_p(w_1)=d$, $\eta_p(w_i)=d-1$ for $2
\le i \le p$ and $\eta_p(e_i)=0$ for $1 \le i \le p$.

\begin{claim}
\label{cl2} Let $\eta'_p=\eta_p$ except that
$\eta'_p(w_2)=\eta_p(w_2)-1=d-2$ and $\eta'_p(e_1)=\eta_p(e_1)+1=1$.
$$\per(A_{H_p}(\eta'_p)) \ne 0 \pmod{d+1}.$$
\end{claim}
\begin{proof}
To prove this claim, we write the column $A_{H_p}(e_1)$ as
$A_{H_p}(w_1)+A_{H_p}(w_2)$. By linearity of permanent with respect
to columns, $$\per(A_{H_p}(\eta'_p)) = \per(A') + \per(A''),$$ where
$A' $ is the matrix obtained from $A_{H_p}(\eta'_p)$ by replacing
the column $A_{H_p}(e_1)$ with $A_{H_p}(w_1)$, and $A'' $ is the
matrix obtained from $A_{H_p}(\eta_p)$ by replacing  the column
$A_{H_p}(e_1)$ with $A_{H_p}(w_2)$. Thus $A'' = A_{H_p}(\eta_p)$ and
$A'$ contains $d+1$ copies of the column
 $A_{H_p}(w_1)$.  Hence $\per(A') = 0 \pmod{d+1}$.
Therefore, $\per(A_{H_p}(\eta'_p)) = \per(A'') \pmod{d+1} =
\per(A_{H_p}(\eta_p)) \pmod{d+1} \ne 0 \pmod{d+1}$. This completes
the proof of Claim \ref{cl2}.
\end{proof}

Now in $A_{H_p}(\eta'_p)$, we re-write the columns as follows:
\begin{eqnarray*}
A_{H_p}(w_3) &=& A_{H_p}(e_2)-A_{H_p}(w_2),\\
A_{H_p}(w_4) &=& A_{H_p}(e_3)-A_{H_p}(w_3),\\
\ldots\\
A_{H_p}(w_p) &=& A_{H_p}(e_{p-1})-A_{H_p}(w_{p-1}),\\
A_{H_p}(w_1) &=& A_{H_p}(e_p)-A_{H_p}(w_p).
\end{eqnarray*}
By using these expressions, in the matrix $A_{H_p}(\eta'_p)$, the
column $A_{G'}(z)$ occurs at most $d-1$ times for each $z \in V(G)$
and the column  $A_{G'}(z)$ occurs at most once for each $z \in
E(G)$. For each $z \notin V(G) \cup E(G)$,  the column occurs $0$
times. By Lemma \ref{subgraph}, $G$ is $(d,2)$-choosable

Assume next that $s \ge 2$. Then the  path $P'=( w_1, w_2, \cdots,
w_s)$ connect $w_1$ to a cycle  $C=(w_s, w_{s+1}, \ldots, w_p)$. In
$A_{H_p}(\eta_p)$, write one copy of $A_{H_p}(w_1)$ as
$$A_{H_p}(e_1)-A_{H_p}(e_2)+\ldots +(-1)^sA_{H_p}(e_{s-1})+(-1)^{s+1}A_{H_p}(w_s).$$
%express one copy of $w_1$ as $e_1-e_2+e_3- \ldots +(-1)^{s}e_{s-1}+(-1)^{s-1}w_{s}$.
By using this expression and by linearity of permanent with respect
to columns, we obtain an index function $\eta'$ of $H_p$ in which
$\eta'(z) \le d-1$ for all $z \in V(G)$ except that possibly
$\eta'(w_s)=d$, and $\eta'(z) \le 1$ for all $z \in E(G)$,  such
that $\per(A_{H_p}(\eta')) \ne 0 \pmod{d+1}$. Moreover, for this
index function $\eta'$, we have $\eta'(w_i)=d-1$ for $s+1 \le i \le
p$, $\eta'(e_i)=0$ for $s \le i \le p$. This is the same as the
$s=1$ case, and the proof is complete.
\end{proof}
\begin{corollary}
\label{2deg} Every $2$-degenerate graph  is $(2,2)$-choosable.
\end{corollary}

\section{Graphs with bounded maximum average degree}

 The {\em average
degree} $\bar{d}(G)$ of $G$ is $\bar{d}(G)=\frac{2|E(G)|}{|V(G)|}$.
The {\em maximum average degree} of $G$, denoted by ${\rm mad}(G)$,
is defined as ${\rm mad}(G)=\max\{\bar{d}(H):H\subseteq G\}$. This
section
%TL to add the below?(but write $(k, 2)$-choosability here seems not good since the later part write $(k+1, 2)$-choosable)
%====
%gives another bound for $(k, 2)$-choosability which
%===
proves that if ${\rm mad}(G) \le 2k$ for some integer $k$, then $G$
is $(k+1, 2)$-choosable.

\begin{lemma}
\label{orientation} Assume $D$ is an orientation of a graph $G$, and
$\eta$ is the index function defined as
 $\eta(v)=d_D^+(v)$ for every vertex $v$ and $\eta(e)=1$ for every edge $e$. Then $\eta$ is a non-singular index function
 of $G$.
 \end{lemma}
 \begin{proof}
 First we prove that the lemma is true if $D$ is an acyclic orientation.
 In this case, we prove that the index function $\eta$ defined as $\eta(v)=d^+_D(v)$ for each vertex $v$ and
 $\eta(e)=0$ for each edge $e$ is a valid index function with $\per(A_G(\eta)) \ne 0$.

 Assume this is not true and $G$ is a minimum counterexample. As $D$ is acyclic,
 there is a source vertex $v$. By the minimality of $G$,  the restriction $\eta'$ of $\eta$ to $G-v$
 is a non-singular index function of $G-v$.    We extend the matrix $A_{G-v}(\eta')$  to $A_G(\eta)$ by adding
 $d_G(v)$ rows indexed by edges incident to $v$, and  adding $d_G(v)=d_D^+(v)$
 copies of the column $A_G(v)$. Then $\per(A_G(\eta)) = d_G(v) ! \per(A_{G-v}(\eta')) \ne 0$.

 Next we consider the case that $D$ is an arbitrary orientation. Let $D'$ be an acyclic orientation of $G$.
 Let $\eta'$ be the index function defined as $\eta'(v)=d^+_{D'}(v)$ for each vertex $v$ and
 $\eta'(e)=0$ for each edge $e$. By the previous paragraph, $\per(A_G(\eta')) \ne 0$.
 For each directed edge $e=(u, v)$ of $D'$ that is oriented differently in $D$,
 we replace a copy of the column $A_G(u)$ by the linear combination $A_G(e)-A_G(v)$.
 Note that the matrix is not changed, because $A_G(u)=A_G(e)-A_G(v)$.
 However, in such linear combinations of the columns of $A_G(\eta')$, for each edge $e$, $A_G(e)$ occurs at most once,
 and for each vertex $v$, $A_G(v)$ occurs at most $d_D^+(v)$ times. Therefore the index function defined as
 $\eta(v)=d_D^+(v)$ for every vertex $v$ and $\eta(e)=1$ for every edge $e$ is a non-singular index function
 of $G$.
 \end{proof}

\begin{corollary}
If ${\rm mad}(G) \le 2k$, then $G$ is $(k+1, 2)$-choosable. In
particular, planar graphs are $(4,2)$-choosable and planar bipartite
graphs are $(3,2)$-choosable.
\end{corollary}
\begin{proof}
It is well-known that if $G$ has maximum average degree at most
$2k$, then $G$ has an orientation with maximum out-degree at most
$k$. Therefore the index function $\eta$  defined as
 $\eta(v)=k$ for every vertex $v$ and $\eta(e)=1$ for every edge $e$ is a non-singular index function
 of $G$. It follows from the argument in the introduction that $G$ is $(k+1,2)$-choosable.
 \end{proof}

\end{document}